\documentclass[12pt]{amsart}

\usepackage[english]{babel}
\usepackage[utf8x]{inputenc}
\usepackage{geometry}
\usepackage{graphicx}				
\usepackage[colorlinks=true,linkcolor=blue,urlcolor=blue,citecolor=blue]{hyperref}

\usepackage{mathrsfs}
\usepackage{amssymb}
\usepackage{amsthm}
\usepackage{amsmath,amsfonts}
\usepackage{comment,enumitem}
\usepackage{color}

 \newcommand{\diam}{\operatorname{diam}}  
 \newcommand{\aut}{\operatorname{Aut}}  

 \newtheorem{thm}{Theorem}
\newtheorem{prop}[thm]{Proposition}
\newtheorem{cor}[thm]{Corollary}
\newtheorem{lem}[thm]{Lemma}

\theoremstyle{definition}

\newtheorem{remark}[thm]{Remark}

\newtheorem{example}[thm]{Example}

\makeatletter
\@namedef{subjclassname@2020}{%
  \textup{2020} Mathematics Subject Classification}
\makeatother

\begin{document}

\title[Doubling constants and spectral theory on graphs]{Doubling constants and spectral theory on graphs}

\author[E. Durand-Cartagena]{Estibalitz Durand-Cartagena$^*$}
\address{Departamento de Matem\'atica Aplicada, ETSI Industriales, UNED\\
28040 Madrid, Spain.} 
\email{edurand@ind.uned.es }

\author[J. Soria]{Javier Soria$^\dagger$}
\address{
Instituto de Matem\'atica Interdisciplinar (IMI); Departamento de An\'alisis Ma\-te\-m\'a\-ti\-co y Matem\'atica Aplicada, Universidad Complutense de Madrid\\
28040 Madrid, Spain.}
\email{javier.soria@ucm.es}

\author[P. Tradacete]{Pedro Tradacete$^{\ast\ast}$}
\address{Instituto de Ciencias Matem\'aticas (CSIC-UAM-UC3M-UCM)\\
Consejo Superior de Investigaciones Cient\'ificas\\
C/ Nicol\'as Cabrera, 13--15, Campus de Cantoblanco UAM\\
28049 Madrid, Spain.}
\email{pedro.tradacete@icmat.es}

\thanks{$^*$E. Durand-Cartagena was partially supported by MINECO (Spain), project PGC2018-097286-B-I00 and  the grant 2021-MAT11 (ETSI Industriales, UNED)}

\thanks{$^\dagger$J. Soria  was partially supported by grants PID2020-113048GB-I00 funded by MCIN/AEI/ 10.13039/501100011033, and Grupo UCM-970966.}

\thanks{$^{\ast\ast}$P. Tradacete was partially supported by grants PID2020-116398GB-I00, MTM2016-76808-P, MTM2016-75196-P and CEX2019-000904-S funded by MCIN/AEI/ 10.13039/501100011033.}

\begin{abstract}
We study the least doubling constant  among all possible doubling measures defined on a (finite or infinite) graph $G$. We show that this constant can be estimated from below by  $ 1+ r(A_G)$, where $r(A_G)$ is the spectral radius of the adjacency matrix of  $G$, and study when both quantities coincide. We also illustrate how amenability of the automorphism group of a graph can be related to finding doubling minimizers. Finally, we give a complete characterization of graphs with doubling constant smaller than 3, in the spirit of Smith graphs.
\end{abstract}


\date{\today}

\maketitle


\section{Introduction}

The aim of this paper is to exhibit the connection between combinatorial aspects of graph theory and geometric measure theory, by studying doubling measures on graphs. In particular, we will establish a new connection between spectral graph theory and doubling constants of certain measures on a graph.

We will deal with locally finite, unoriented, connected, simple graphs (without loops nor multiple edges), and we will consider both cases of finite and infinite sets of vertices. Such a graph $G$ with vertices $V_G$ and edges $E_G$ can be considered as a metric space in a standard way by means of the path distance: Given vertices $x,y\in V_G$, a path joining $x$ to $y$ is a collection of edges of the form $\{x_{i-1},x_i\}_{i=1}^k\subset E_G$, with $x_0=x$ and $x_k=y$. In this case, we say that the path has length $k$. Thus, for $x,y\in V_G$, the distance $d_G(x,y)$ is defined as the smallest possible length of a path joining $x$ to $y$.

A measure $\mu$ on a graph $G$ will always be induced by a positive weight function $\mu:V_G\rightarrow (0,\infty)$, and despite the slight abuse of terminology, we will denote $\mu(A)=\sum_{v\in A}\mu(v)$ for any set $A\subset V_G$. Such $\mu$  is said to be a \emph{doubling measure} whenever
$$
C_\mu=\sup_{v\in V_G, r\geq0}\frac{\mu(B(v,2r))}{\mu(B(v,r))}<\infty.
$$
Here, $B(v,r)=\{w \in V_G: d_G(v,w)\leq r\}$ denotes the closed ball of center $v\in V_G$ and radius $r\geq0$. The number $C_\mu$ is called the \emph{doubling constant} of the measure $\mu$.  Note that if $V_G$ is finite, every measure on $G$ is doubling for an appropriate constant.

Associated to a general metric space $(X,d)$, the following invariant was introduced in \cite{ST}
$$
C_{(X,d)}=\inf\Big\{\sup_{x\in X, r\geq0}\frac{\mu(B(x,2r))}{\mu(B(x,r))}: \mu \text{ doubling measure on } X\Big\},
$$
which will be referred to as the \emph{least doubling constant} of $(X,d)$. It was also shown in \cite{ST} that, if $X$ supports a doubling measure and contains more than one point, then $C_{(X,d)}\geq 2$. This invariant is somehow related to metric dimension theory and our purpose is to study its properties in the context of graphs.

For simplicity, we will denote $C_G=C_{(V_G,d_G)}$. In \cite{ST}, we already computed this invariant for certain families of finite graphs. Namely, let $K_n$ denote the complete graph with $n$ vertices, $S_n$ be the star-graph with $n+1$ vertices (one vertex of degree $n$ and $n$ leaves), and $C_n$ be the cycle-graph with $n$ vertices: For $n\geq 3$, we have 
\begin{equation}
C_{K_n}=n,\quad \quad C_{S_n}=1+\sqrt{n}, \quad \quad C_{C_n}=3.
\end{equation}



We will here compute $C_G$ for other families of graphs, including for instance complete bipartite graphs $K_{m,n}$, wheel graphs $W_n$, friendship graphs or cocktail party graphs. 

It is somehow surprising that, for all the examples listed so far, it turns out that $C_G$ coincides with $1+r(A_G)$, where $r(A_G)$ is the spectral radius (equivalently, the largest eigenvalue) of the adjacency matrix of the graph. This observation led us to look for the relation between doubling constants and spectral graph theory, which is the main goal of this work. We will prove in Theorems \ref{thm:doubling spectra} and \ref{thm:doubling spectra infinite}, that, in general, $$C_G\geq 1+r(A_G).$$ In the case of finite graphs, Proposition \ref{lemachorra} actually provides a characterization of when $C_G=1+r(A_G)$, in terms of the measure induced by the Perron eigenvector of $A_G$. From a more geometric point of view, Proposition \ref{p:diam2} yields in particular that every graph $G$ with diameter 2 also satisfies $C_G=1+r(A_G)$.

The paper is organized as follows: in Section~\ref{sec2} we start by showing, in Proposition~\ref{p:DMconvex}, the existence of doubling minimizers; that is, those measures $\mu$ for which $C_\mu=C_G$, as well as some stability properties of this set of measures. Section~\ref{sec3} is devoted to the proof of $C_G\geq 1+r(A_G)$ both for finite and infinite graphs, and some direct consequences related to monotonicity of the doubling constants and the chromatic number of a graph. In Section~\ref{sec4}, Theorem~{\ref{t:symmetric}} establishes the main correspondence between the amenability of the group of automorphisms of $G$ and the existence of invariant minimizers. This allows us to consider only symmetric measures when computing $C_G$, thus reducing the complexity of the problem in certain cases. We consider in Section~\ref{sec5} the question of whether $C_G$ can be determined by the spectra of $G$ and do a thorough study for several relevant cases of graphs with small diameter. We also provide explicit examples where $C_G>1+r(A_G)$. Finally, motivated by the works of J. H. Smith \cite{S} and D. Cvetkovi\'c and I. Gutman \cite{CG} on graphs with small spectral radius, in Section~\ref{sec6}, we completely characterize  those graphs $G$ with doubling constant $C_G\leq 3$ (see Corollaries~\ref{c:finiteCG3} and~\ref{c:infiniteCG3}).

Let us finally mention that in the companion paper \cite{DST}, we have focused our analysis on the not that simple case of path graphs. In particular, in \cite{DST} it is shown that 
$$
1+2\cos\Big(\frac\pi{n+1}\Big)\leq C_{L_n}< 3,
$$
where $L_n$ denotes the path graph of $n$ vertices, and $C_{\mathbb Z}=C_{\mathbb N}=3$. Also, a careful analysis of the structure of the set of doubling minimizers for path graphs is provided. We refer to \cite{BM} for standard notations and basic facts on graph theory, as well as \cite{CRS} for a comprehensive introduction to spectral graph theory.

\section{Basic properties of the least doubling constant of a graph}\label{sec2}

A couple of comments on the kind of graphs we will be dealing with are in order. First recall that a metric space is called \textsl{metrically doubling} if there is a constant $K>0$ such that for every $r>0$ every ball of radius $r$ can be covered by at most $K$ balls of radius $r/2$ (a space with this property is also called a homogeneous space, cf. \cite{CW}). In particular, it is easy to check that every metric space supporting a doubling measure must be metrically doubling \cite{CW}. Refining earlier results for compact metric spaces given in \cite{VK}, it was shown in \cite{LuSa} that for complete metric spaces being metrically doubling is equivalent to supporting non-trivial doubling measures. Since a graph is always a discrete space (the distance between two distinct points cannot be smaller than 1) in particular every graph is a complete metric space and the previous result applies. In particular, if a graph supports a doubling measure then it must have uniformly bounded degree: $\Delta_G=\sup_{v\in V_G} d_v<\infty$. Also, since we are only dealing with connected graphs, note that the number of vertices must be at most countable.

Given a graph $G$, it is easy to see that
\begin{equation*}
C_G=\inf_{\mu}\sup\Big\{\frac{\mu(B(v,2k+1))}{\mu(B(v,k))}: v\in V_G,\,k\in\mathbb N\cup\{0\} \Big\},
\end{equation*}
the infimum being taken over all doubling measures $\mu$ on $G$. Also, given $G$ we can consider its diameter: $\diam(G)=\sup\{d_G(v,w):v,w\in V_G\}$. If $\diam(G)<\infty$, we actually have
\begin{equation}\label{eq:C_G}
C_G=\inf_{\mu}\sup\Big\{\frac{\mu(B(v,2k+1))}{\mu(B(v,k))}: v\in V_G,\,k\in\mathbb Z, \,0\leq k\leq \Big\lceil \frac{\diam(G)-1}{2}\Big\rceil \Big\},
\end{equation}
where $\lceil s\rceil=\min\{n\in\mathbb Z:s\leq n\}$.

Given a doubling measure $\mu$ on $G$, it will be convenient to consider the \textit{restricted doubling constants} associated to each radius $k\in\mathbb N\cup\{0\}$:
$$
C_\mu^k=\sup_{v\in V_G}\frac{\mu(B(v,2k+1))}{\mu(B(v,k))}.
$$
Also, let 
\begin{equation*} 
C_G^k=\inf_{\mu} C_\mu^k,
\end{equation*}
where the infimum is taken over all doubling measures $\mu$ on $G$. 

Clearly, for a fixed measure $\mu$ on $G$, we have $C_\mu=\sup_{k\in\mathbb N\cup\{0\}} C_\mu^k$, and in the case when $\diam(G)<\infty$, we actually have $C_\mu=\sup_{0\leq k< \lceil \frac{\diam(G)-1}{2}\rceil+1} C_\mu^k$. It is easy to check that
\begin{equation*} 
C_G\geq \sup_{k\in\mathbb N\cup\{0\}} C_G^k,
\end{equation*}
but we will later see that the inequality can in fact be strict (Proposition \ref{p:tripar}).

If the infimum in \eqref{eq:C_G} is attained at some $\mu$, such a measure will be called a \emph{doubling minimizer} for $G$. Given a graph $G$, with $C_G<\infty$, let us denote the set of doubling minimizers by 
$$
DM(G)=\{\mu:\text{ doubling measure on }G \text{ and } C_\mu=C_{G}\},
$$
and similarly, for each $k\in\mathbb N\cup\{0\}$, let
$$
DM^k(G)=\{\mu:\text{ doubling measure on }G \text{ and } C_\mu^k=C_{G}^k\}.
$$

We will see next that the sets $DM(G)$ and $DM^k(G)$ are always non-empty convex cones. Before showing this, the following elementary lemma will be convenient throughout (see \cite{DST}).

\begin{lem}\label{l:holder}
Let $(\alpha_j)_{j=1}^m, (\beta_j)_{j=1}^m$ be positive real scalars. We have that
$$
\frac{\sum_{j=1}^m\alpha_j}{\sum_{j=1}^m\beta_j}\leq\max_{1\leq j\leq m}\Big\{\frac{\alpha_j}{\beta_j}\Big\}.
$$
Moreover, equality holds if and only if $\frac{\alpha_i}{\beta_i}=\frac{\alpha_j}{\beta_j}$, for every $1\leq i,j\leq m$.
\end{lem}

\begin{prop}\label{p:DMconvex}
If  $G$ is a graph, with $C_G<\infty$, then $DM(G)$ and $DM^k(G)$ are non-empty convex cones for every $k\in\mathbb N\cup\{0\}$.
\end{prop}

\begin{proof}
We will provide the proof in the case when $G$ is infinite, leaving to the reader the adaptations to the simpler case of a finite graph. Let $(v_j)_{j\in\mathbb N}$ be an enumeration of $V_G$. We show first that $DM(G)\neq \emptyset$. For every $n\in\mathbb N$, let $\nu_n$ be a measure on $G$ such that $C_{\nu_n}\leq C_G+\frac1n$. Set $\mu_n=\frac{\nu_n}{\nu_n(v_1)}$. Thus, we have $C_{\nu_n}=C_{\mu_n}\leq C_G+\frac1n$ and $\mu_n(v_1)=1$, for every $n\in\mathbb N$.

We claim that, for every $v\in V_G$, we have that  $\sup_n \mu_n(v)<\infty$. Indeed, this will follow by induction on $m=d(v,v_1)$: clearly, the statement holds for $m=0$; that is, when $v=v_1$; now suppose $\sup_n \mu_n(v)<\infty$, for every $v\in V_G$, with $d(v,v_1)\leq m$, and let $w\in V_G$ with $d(w,v_1)=m+1$; we can pick $v\in V_G$ with $d(v,v_1)=m$ and $d(v,w)=1$; hence, we have
$$
\mu_n(w)\leq \mu_n(B(v,1))\leq C_G^0 \mu_n(v),
$$
and the claim follows.

Therefore, we can take an infinite set $A_1\subset \mathbb N$ such that if we write $A_1=\{n^1_i:i\in\mathbb N\}$ with $n^1_i<n^1_{i+1}$ then $\lim_{i\rightarrow\infty} \mu_{n^1_i}(v_1)$ exists. Inductively, for each $j\in\mathbb N$ we can take an infinite set $A_j\subset \mathbb N$ with $A_{j+1}\subset A_j$ and $A_j=\{n^j_i:i\in\mathbb N\}$ so that $\lim_{i\rightarrow\infty}\mu_{n^j_i}(v_j)$ exists for every $j$.

Now, let us define $\mu$ on $G$ by
$$
\mu(v_j)=\lim_{i\rightarrow \infty}\mu_{n_i^j}(v_j).
$$
Note that $\mu$ is well defined because of the claim. It is straightforward to check now that $C_\mu=C_G$. 

Since for every positive scalar $\alpha$, $C_{\alpha\mu}=C_\mu$, in order to see that  $DM(G)$ is a convex cone, it is enough to show that if $\mu_1,\mu_2\in DM(G)$, then $\mu_1+\mu_2\in DM(G)$.  Thus, suppose $C_{\mu_i}=C_G$, for $i=1,2$ and let $\mu=\mu_1+\mu_2$. For every $v\in V_G$ and $k\in\mathbb Z$ with $0\leq k\leq \Big\lceil \frac{\diam(G)-1}{2}\Big\rceil$, by Lemma~\ref{l:holder}, we have
\begin{align*}
\frac{\mu(B(j,2k+1))}{\mu(B(j,k))}&=\frac{\mu_1(B(j,2k+1))+\mu_2(B(j,2k+1))}{\mu_1(B(j,k))+\mu_2(B(j,k))}\\
&\leq \max\Big\{\frac{\mu_1(B(j,2k+1))}{\mu_1(B(j,k))},\frac{\mu_2(B(j,2k+1))}{\mu_2(B(j,k))}\Big\}\\
&\leq \max\{C_{\mu_1},C_{\mu_2}\}=C_{G}.
\end{align*}
It then follows  that $C_\mu\leq C_G$. By definition of $C_G$, it follows that $C_\mu=C_G$ as claimed. 

A similar argument also yields that $DM^k(G)$ is a non-empty convex cone for every $k\in\mathbb N\cup\{0\}$.
\end{proof}

\begin{remark}
For finite graphs, we will see that $DM^0(G)$ consists of a single measure (up to multiplicative constants), which is actually given by the Perron eigenvector of the adjacency matrix $A_G$ (see Theorem \ref{thm:doubling spectra}). However, it is worth noting that in general each of the sets $DM(G)$ and $DM^k(G)$ can contain non-proportional measures. This is for instance the case when $G=\mathbb N$ (that is $V_\mathbb N=\mathbb N$ and $E_\mathbb N=\{\{j,j+1\}:j\in\mathbb N\}$), as shown in  \cite{DST}: For every $\frac{1}{2}\leq \alpha\leq 1$, the measure $\mu_\alpha$ given by 
$$
\mu_\alpha(j)=
\left\{
\begin{array}{ccc}
 \alpha, &   & j=1,  \\
 1, &   &   j>1,
\end{array}
\right.
$$
satisfies $C_{\mu_\alpha}=C_\mathbb N=C_{\mu_\alpha}^0=C_\mathbb N^0=3$.
\end{remark}

In the setting of infinite graphs, the relation between boundedness of maximal operators and the counting measure being doubling was explored in \cite{ST2}. As we have   mentioned above, the existence of a doubling measure on a graph, implies that $\Delta_G<\infty$, which could be interpreted as the counting measure being locally doubling. The following example illustrates that this cannot be further extended.

\begin{example}
A graph supporting doubling measures where the counting measure is not doubling.
\end{example}

\begin{proof}
We will consider the graphs $\mathbb Z\times \mathbb Z$ and $\mathbb N$ joined by a vertex: Let $V_G=\{(m,n,0):m,n\in\mathbb Z\}\cup\{(0,0,p):p\in\mathbb N\}$ and edges connecting $(m,n,0)$ with $(m,n\pm1,0)$ and $(m\pm1,n,0)$, and $(0,0,p)$ connected with $(0,0,p-1)$ for every $m,n\in\mathbb Z$ and $p\in\mathbb N$.

It is straightforward to check that $G$ is metrically doubling: every ball of radius $r$ is contained in the union of at most 5 balls of radius $r/2$. Since $G$ is a discrete metric space, it follows from \cite{LuSa} that there exist doubling measures on $G$. On the other hand, the counting measure $|\cdot|$ is not doubling: $|(B(0,0,k),k)|=2k+1$ while $|B((0,0,k),2k+1)|\approx k^2$.
\end{proof}

\section{Doubling constants meet spectral graph theory}\label{sec3}

Recall that the adjacency matrix of a graph $G$ is the (possibly infinite) matrix $A_G$ indexed by $V_G\times V_G$, whose $(i,j)$ entry is $1$ if the edge $\{i,j\}\in E_G$ and $0$ otherwise. We will first deal with the case of finite graphs:

\begin{thm}\label{thm:doubling spectra}
For every finite graph $G$, it holds that
$$
C_G^0=1+r(A_G),
$$
where $r(A_G)$ denotes the spectral radius (which coincides with the largest eigenvalue) of the adjacency matrix of $G$. Moreover, in that case $DM^0(G)$ consists of a unique minimizer (up to multiplicative factors) which is the measure $\mu$ given by the Perron eigenvector corresponding to $r(A_G)$.
\end{thm}

\begin{proof}
Since all the entries in the matrix $A_G$ are non-negative and $G$ is connected, by Perron-Frobenius theorem (cf. \cite[Theorem 8.26]{AA}), $\lambda_1(A_G)=r(A_G)$, the spectral radius of $A_G$, is an eigenvalue with one dimensional eigenspace generated by a vector $x_G=(a_1,\ldots,a_n)$ with strictly positive entries, and all the remaining eigenvalues are (real and) strictly smaller than $r(A_G)$.

Let $\mu(i)=a_i$ for $1\leq i\leq n$. Note that for each $1\leq i\leq n$ we have
$$
\mu(B(i,1))=\mu(i)+\sum_{\{i,j\}\in E_G} \mu(j)=a_i+\sum_{\{i,j\}\in E_G} a_j=a_i+(A_G x_G)_i=a_i(1+\lambda_1(A_G)).
$$
Therefore,
$$
C_G^0\leq \sup_{i\in V_G}\frac{\mu(B(i,1))}{\mu(B(i,0))}=1+\lambda_1(A_G).
$$

For the converse inequality, we claim that given any strictly positive $y\in \mathbb R^n$ we must have 
\begin{equation}\label{eq:spectralradius}
\sup_{1\leq i\leq n}\frac{(A_G y)_i}{y_i}\geq \lambda_1(A_G).
\end{equation}

Indeed, suppose this is not the case. Therefore, there exists $y\in\mathbb R^n$ with strictly positive entries such that 
$$
r=\sup_{1\leq i\leq n}\frac{(A_G y)_i}{y_i}< \lambda_1(A_G)=r(A_G).
$$
We can now define a norm in $\mathbb R^n$ as follows: given $x\in\mathbb R^n$
$$
\|x\|_0=\inf\{\lambda>0:|x|\leq \lambda y\}.
$$
This is indeed a norm because all the entries of $y$ are strictly positive. Note that if $0\leq x\leq\lambda y$ we have that
$$
A_G x\leq A_G \lambda y\leq r \lambda y.
$$
Iterating, we get that if $0\leq x\leq\lambda y$, then for every $j\in\mathbb N$ we have that
$$
A_G^j x\leq r^j\lambda y.
$$
Therefore, we have that
$$
\|A_G^j x\|_0\leq r^j \|x\|_0,
$$
so that $\|A_G^j\|\leq r^j$ with respect to $\|\cdot\|_0$. Now, using Gelfand's formula for the spectral radius (cf. \cite[Theorem 6.12]{AA}) it would follow that
$$
r(A_G)=\lim_{j \rightarrow \infty} \|A_G^j\|^{\frac1j}\leq r.
$$
This is a contradiction, so we must have \eqref{eq:spectralradius}.

Hence, given any doubling measure $\nu$ on $G$, take $y\in\mathbb R^n$ such that $y_i=\nu(\{i\})>0$. We have that
$$
\sup_{i\in V_G}\frac{\nu(B(i,1))}{\nu(B(i,0))}=\sup_{i\in V_G}\frac{y_i+ (A_G y)_i}{y_i}\geq 1+ \lambda_1(A_G).
$$
This finishes the proof.
\end{proof}

\begin{remark}
The previous result is somehow reminiscent of a classical theorem due to H. Wiedlandt \cite{Wie} about the spectral radius of non-negative irreducible matrices (see also \cite{Friedland} for further developments).
\end{remark}

Recall that $G_1$ is a subgraph of $G_2$, if $V_{G_1}\subset V_{G_2}$ and $E_{G_1}\subset E_{G_2}$. Note that if $G_1$ is a subgraph of $G_2$, then the adjacency matrices (up to completion by zeros) satisfy 
$$
A_{G_1}\leq A_{G_2}
$$
with the order given entry-wise. It follows that for every $k\in\mathbb N$, we have $A_{G_1}^k\leq A_{G_2}^k$, so by Theorem \ref{thm:doubling spectra} and Gelfand's formula for the spectral radius, we have
$$
C^0_{G_1}=1+r(A_{G_1})=1+\lim_{k\rightarrow \infty} \|A_{G_1}^k\|^{\frac{1}{k}}\leq1+\lim_{k\rightarrow \infty} \|A_{G_2}^k\|^{\frac{1}{k}}=1+r(A_{G_2})=C^0_{G_2}.
$$
The following result is even more informative:

\begin{cor}\label{c:C0monotone} 
Suppose $G_1$ is a (proper) subgraph of a finite graph $G_2$, then $C^0_{G_1}< C^0_{G_2}$.
\end{cor}

\begin{proof}
Use Theorem \ref{thm:doubling spectra} together with the fact that if $G_1$ is a (proper) subgraph of $G_2$, then $r(A_{G_1})<r(A_{G_2})$ (cf. \cite[Propositions 1.3.9 \& 1.3.10]{CRS}). 
\end{proof}

Our aim now is to provide a version of Theorem \ref{thm:doubling spectra} in the context of infinite graphs. In order to do so, we need some notation and lemmas first. Given an infinite graph $G$, we can consider the infinite adjacency matrix $A_G$ whose $(v,w)$ entry is given, for $v,w\in V_G$, by
$$
A_G(v,w)=
\left\{
\begin{array}{ccl}
 1, &   & \text{if }\{v,w\}\in E_G,  \\
 0, &   &   \text{otherwise.}
\end{array}
\right.
$$ 
Also, for $1\leq p\leq\infty$, let $\ell_p(V_G)$ denote the space of functions $f:V_G\rightarrow \mathbb R$ for which the norm $\|f\|_p=(\sum_{v\in V_G} |f(v)|^p)^{1/p}<\infty$ (as usual, $\|f\|_\infty=\sup_{v\in V_G}|f(v)|$). Let us recall the following well-known fact (cf. \cite[Theorem 4]{BR}):

\begin{lem}
Let $G$ be a graph supporting non-trivial doubling measures. For every $1\leq p\leq \infty$, $A_G:\ell_p(V_G)\rightarrow \ell_p(V_G)$ defines a bounded linear positive operator.
\end{lem}

\begin{proof}
For $v\in V_G$, let $e_v:V_G\rightarrow \mathbb R$ be given by $e_v(v)=1$ and $e_v(w)=0$ for $w\neq v$. Associated to the adjacency matrix we set a linear operator given by $A_G e_v=\sum_{w\in V_G} A_G(v,w)e_w$. Since $G$ supports non-trivial doubling measures, in particular we have that $\Delta_G<\infty$. Thus, for every $v\in V_G$ we have that
$$
\|A_G e_v\|_1=\sum_{(v,w)\in E_G} 1\leq \Delta_G.
$$
Hence, $A_G:\ell_1(V_G)\rightarrow \ell_1(V_G)$ is bounded with $\|A_G\|\leq \Delta_G$. Also, if $1_G$ denotes the constant 1 function on $V_G$, we have
$$
\|A_G 1_G\|_\infty=\sup_{v\in V_G}\sum_{(v,w)\in E_G} 1\leq \Delta_G.
$$
Therefore, by positivity of $A_G$ it follows that $A_G:\ell_\infty(V_G)\rightarrow \ell_\infty(V_G)$ is bounded with $\|A_G\|\leq \Delta_G$. The conclusion follows by standard interpolation.
\end{proof}

In order to extend Theorem \ref{thm:doubling spectra} to the infinite setting, it is natural to consider the extensions of Perron-Frobenius theorem to the infinite dimensional context. These extensions, like the Krein-Rutman Theorem \cite[Theorem 7.10]{AA}, usually require compactness for an operator to have the spectral radius as an eigenvalue with positive eigenvector. However, the adjacency operator $A_G$ need not be compact in general, so a different approach is required. In this direction, it was shown by B. Mohar in \cite{Mohar} that if $r(A_G)$ denotes the spectral radius of the operator $A_G:\ell_2(V_G)\rightarrow \ell_2(V_G)$, then
\begin{equation}\label{eq:Mohar}
r(A_G)=\sup\{r(A_F):F \text{ a finite subgraph of }G\}.
\end{equation}

We need an analogue of this fact for the restricted doubling constant as follows:

\begin{prop}\label{p:C0supsubgraph}
Let $G$ be an infinite graph supporting doubling measures. Then
$$C_G^0=\sup\{C_F^0:F \text{ a finite subgraph of }G\}.$$
\end{prop}

\begin{proof}
By Proposition \ref{p:DMconvex}, we can consider a measure $\mu$ on $G$ such that $C_G^0=C_\mu^0$. For each finite subgraph $F$ of $G$, we can take $\mu_F$ to be the restriction of $\mu$ to $F$. We have that
$$
C_F^0\leq C_{\mu_F}^0=\sup_{v\in V_F}\frac{\mu_F(B(v,1))}{\mu_F(v)}\leq \sup_{v\in V_F}\frac{\mu(B(v,1))}{\mu(v)}\leq C_\mu^0=C_G^0.
$$
Taking the supremum over all finite subgraphs of $G$, we get
$$
\sup\{C_F^0:F \text{ a finite subgraph of }G\}\leq C_G^0.
$$

For the converse inequality, let $(F_n)_{n\in\mathbb N}$ be an increasing sequence (i.e., $F_n\subset F_{n+1}$) of finite subgraphs of $G$ such that $G=\bigcup_n F_n$. Let also $(v_k)_{k\in\mathbb N}$ be an enumeration of $V_G$ with $v_1\in F_1$. Since $G=\bigcup_n F_n$ for each $k\in\mathbb N$, we can consider 
$$
n_k=\min\{n\in\mathbb N: v_k\in F_n\}.
$$ 
For each $n\in\mathbb N$, let $\mu_n$ be the measure on $F_n$ associated to the Perron eigenvector of the adjacency matrix $A_{F_n}$ so that
$$
C_{F_n}^0=C_{\mu_n}^0,
$$
and without loss of generality, let us assume $\mu_n(v_1)=1$. Note that by Corollary \ref{c:C0monotone} $C_{F_n}^0$ is a monotone increasing sequence. We claim that for every $k\in\mathbb N$, 
$$
M_k=\sup_{n\geq n_k} \mu_n(v_k)<\infty, \quad\quad m_k=\inf_{n\geq n_k} \mu_n(v_k)>0.
$$
Indeed, this follows by induction on the distance to $v_1$ exactly as in the proof of Proposition \ref{p:DMconvex}. Similarly, there exists a decreasing family of infinite sets $A_k=\{n_i^k:i\in\mathbb N\}$ (with $n_i^k\geq n_k$ for every $i\in\mathbb N$) so that $\lim_{i\rightarrow\infty} \mu_{n_i^k}(v_k)$ exists for every $k\in\mathbb N$. We can thus define the measure $\mu$ on $V_G$ by
$$
\mu(v_k)=\lim_{i\rightarrow \infty} \mu_{n_i^k}(v_k).
$$

Now, for a fixed $k\in \mathbb N$, let $B_k\subset \mathbb N$ be the finite set such that $B(v_k,1)=\{v_j:j\in B_k\}$ and let $j_k=\max B_k$. Given $\varepsilon>0$, let $\varepsilon_k=\frac{m_k^2\varepsilon}{M_k(\Delta_G+\varepsilon)+\sum_{j\in B_k} M_j}$ and take $m\in\mathbb N$ large enough so that $|\mu(v_j)-\mu_{n_i^{j_k}}(v_j)|\leq\varepsilon_k$ for $i\geq m$ and every $j\in B_k$. We have that
\begin{align*}
\frac{\mu(B(v_k,1))}{\mu(v_k)}&=\frac{\sum_{j\in B_k} \mu(v_j) }{\mu(v_k)}\leq \frac{\sum_{j\in B_k} \mu_{n_m^{j_k}}(v_j) +\varepsilon_k}{\mu_{n_m^{j_k}}(v_k)-\varepsilon_k}\\
&\leq\frac{ \mu_{n_m^{j_k}}(B(v_k,1))}{ \mu_{n_m^{j_k}}(v_k)}+\frac{(\Delta_G \mu_{n_m^{j_k}}(v_k)+ \mu_{n_m^{j_k}}(B(v_k,1)))\varepsilon_k}{ \mu_{n_m^{j_k}}(v_k)( \mu_{n_m^{j_k}}(v_k)-\varepsilon_k)}\\
&\leq C_{\mu_{n_m^{j_k}}}^0+\varepsilon= C_{F_{n_m^{j_k}}}^0+\varepsilon\\
&\leq \sup\{C_F^0:F \text{ a finite subgraph of }G\}+\varepsilon.
\end{align*}
Letting $\varepsilon\rightarrow0$ we get that 
$$
\frac{\mu(B(v_k,1))}{\mu(v_k)}\leq \sup\{C_F^0:F \text{ a finite subgraph of }G\}.
$$ 
Since this holds for every $k\in\mathbb N$, we get that
$$
C_G^0\leq C_\mu^0=\sup_{k\in\mathbb N} \frac{\mu(B(v_k,1))}{\mu(v_k)}\leq \sup\{C_F^0:F \text{ a finite subgraph of }G\}.
$$
\end{proof}

\begin{thm}\label{thm:doubling spectra infinite}
For every graph $G$ such that $\Delta_G<\infty$ we have
$$
C_G^0=1+r(A_G),
$$
where $r(A_G)$ denotes the spectral radius of the adjacency operator $A_G:\ell_2(V_G)\rightarrow \ell_2(V_G)$.
\end{thm}

\begin{proof}
This follows by applying Proposition \ref{p:C0supsubgraph}, Theorem \ref{thm:doubling spectra} and Mohar's identity \eqref{eq:Mohar}:
\begin{align*}
C_G^0&=\sup\{C_F^0:F \text{ a finite subgraph of }G\}\\
&=\sup\{1+r(A_F):F \text{ a finite subgraph of }G\}\\
&=1+r(A_G).
\end{align*}
\end{proof}

Note that the second part of Theorem \ref{thm:doubling spectra} about $DM^0(G)$ consisting of a single ray is no longer true for infinite graphs: take $G=\mathbb N$ and let $\mu$ be the measure on $\mathbb N$ given by $\mu(n)=n$ for $n\in\mathbb N$; it is clear that $C_\mu^0=3=C_{|\cdot|}^0$.

Recall that a graph $G$ is $k$-colorable if there is a function $c:V_G\rightarrow\{1,\ldots,k\}$ such that $c(v)\neq c(w)$ whenever $\{v,w\}\in E_G$. The chromatic number $\chi(G)$ is the least $k$ for which $G$ is $k$-colorable. As a consequence of Wilf's theorem \cite{Wilf} and its infinite dimensional extension \cite{BR} we get the following corollary to Theorems \ref{thm:doubling spectra} and \ref{thm:doubling spectra infinite}.

\begin{cor}
For every graph $G$ it holds that $\chi(G)\leq C_G^0$.
\end{cor}

\section{Symmetric doubling minimizers and amenability of $\aut(G)$}\label{sec4}

In this section we will see that if the action of the automorphism group of the graph is amenable, then the doubling constant can be computed considering only symmetric measures; that is, those invariant under the action of the automorphism group.

Recall that an automorphism of a graph $G = (V_G,E_G)$ is a permutation $\sigma:V_G\rightarrow V_G$, such that the pair $\{u,v\}\in E_G$ if and only if the pair $\{\sigma(u),\sigma(v)\}\in E_G$. Let $\aut(G)$ denote the group of all automorphisms on $G$. From the point of view of metric spaces, the automorphisms are the bijections on the space which preserve the distance, or simply the \textit{global isometries}. 

Let us introduce the following notation: Given a graph $G$, a finite set $F\subset \aut(G)$, and a doubling measure $\mu$ on $G$, let $\mu_F$ be the measure given by
$$
\mu_F(v)=\sum_{\sigma\in F}\mu(\sigma(v)),
$$
for $v\in V_G$. This clearly defines an $F$-invariant measure; that is, $\mu_F(\sigma(v))=\mu_F(v)$ for every $\sigma\in F$.

\begin{lem}\label{l:finite}
Given a graph $G$, a finite set $F\subset \aut(G)$, and a doubling measure $\mu$ on $G$, let $\mu_F$ be as above. For every $k\in\mathbb N\cup\{0\}$ we have $C^k_{\mu_F}\leq C^k_{\mu}$. In particular, $C_{\mu_{F}}\leq C_\mu.$
\end{lem}

\begin{proof}
Note that for any $v\in V_G$, $k\in\mathbb N\cup\{0\}$ and $\sigma\in \aut(G)$ we have
$$
\sigma(B(v,k))=B(\sigma(v),k).
$$
Thus, for $k\in\mathbb N\cup\{0\}$, applying Lemma \ref{l:holder}, we have that
\begin{align*}
\frac{\mu_{F}(B(v,2k+1))}{\mu_{F}(B(v,k))}&=\frac{\sum_{\sigma\in F}\mu(\sigma(B(v,2k+1)))}{\sum_{\sigma\in F}\mu(\sigma(B(v,k)))}\\
&=\frac{\sum_{\sigma\in F}\mu(B(\sigma(v),2k+1))}{\sum_{\sigma\in F}\mu(B(\sigma(v),k))}\\
&\leq\max_{\sigma\in F}\left\{\frac{\mu(B(\sigma(v),2k+1))}{\mu(B(\sigma(v),k))}\right\}\leq C^k_\mu.
\end{align*}
Therefore, we have 
$$
C^k_{\mu_{F}}\leq C^k_\mu,
$$
as claimed. Thus, we also have $C_{\mu_{F}}\leq C_\mu$.
\end{proof}

\begin{remark}
For every graph $G$ whose automorphism group $\aut(G)$ is finite (in particular, for every finite graph), by Proposition \ref{p:DMconvex}, we can take $\mu\in DM(G)$ and consider $\mu_{\aut(G)}$ as above. Lemma \ref{l:finite} yields that $C_G=C_\mu=C_{\mu_{\aut(G)}}$, with $\mu_{\aut(G)}$ invariant under $\aut(G)$. Thus, for these graphs we always have symmetric doubling minimizers. 
\end{remark}

Let us see next that this can be extended to graphs with a subgroup of $\aut(G)$ having an amenable action. We will say that a subgroup $\Gamma$ of $\aut(G)$ is amenable, if there exists a sequence of finite subsets $F_n\subset \Gamma$, for $n\in\mathbb N$, with 
\begin{enumerate}
    \item[{(i)}] $F_n\subset F_{n+1}$,
     \item[{(ii)}]  $\bigcup_n F_n=\Gamma$,
     \item[{(iii)}]  For every $\sigma\in \Gamma$, $\frac{|F_n\cdot \sigma\Delta F_n|}{|F_n|}\underset{n\rightarrow \infty}\longrightarrow 0,$
\end{enumerate}
where $F_n\cdot \sigma=\{\tau\sigma:\tau\in F_n\}$ and $\Delta$ as usual denotes the symmetric difference of two sets (that is, $A\Delta B=(A\cup B)\backslash (A\cap B)$.) A sequence $(F_n)_{n\in \mathbb N}$ as above is usually called a F\o lner sequence in $\Gamma$ \cite{Folner}. Amenable actions on graphs have been thoroughly considered in the literature and have  connections to several areas (see for instance \cite{AL, SW, Woess}).

Also we will say that a measure $\mu$ is bounded on $G$ if $\sup_{v\in V_G} \mu(v)<\infty$.

\begin{thm}\label{t:symmetric}
Let $G$ be a graph and $\Gamma$ an amenable subgroup of $\aut(G)$. If there is a doubling measure $\mu$ which is bounded on $G$, then there is a doubling measure $\mu_\Gamma$ on $G$ such that
\begin{enumerate}
    \item[(i)] $\mu_\Gamma$ is invariant under $\Gamma$,
    \item[(ii)] $C_{\mu_\Gamma}\leq C_\mu$.
\end{enumerate}
\end{thm}

\begin{proof}
As $\Gamma$ is amenable, let $(F_n)_{n\in\mathbb N}$ be a F\o lner sequence in $\Gamma$. Suppose $\mu$ is a bounded doubling measure on $G$ and let $\|\mu\|_\infty=\sup_{v\in V_G}\mu(v)$. For $n\in\mathbb N$, and $v\in V_G$ let
$$
\mu_n(v)=\frac{\mu_{F_n}(v)}{|F_n|}=\frac1{|F_n|}\sum_{\sigma\in F_n}\mu(\sigma(v)).
$$
Let now $\mathcal U$ be an ultrafilter on $\mathbb N$ refining the order filter, and let 
$$
\mu_{\Gamma}(v)=\lim_{n\in\mathcal U}\mu_n(v),
$$
the limit being taken along the ultrafilter $\mathcal U$. Note that $\mu_\Gamma(v)$ is well-defined because $\mu$ is bounded. By Lemma \ref{l:finite}, $C^k_{\mu_n}\leq C^k_\mu$, for every $k\in\mathbb N$, so it clearly follows that 
$$
C_{\mu_\Gamma}\leq C_\mu.
$$
It remains to check that $\mu_\Gamma$ is invariant under $\Gamma$. To this end, fix $v\in V_G$ and $\sigma\in \Gamma$. For every $\varepsilon>0$, we can take $A\in\mathcal U$ such that for every $n\in A$ we have
\begin{enumerate}
    \item[{(a)}] $|\mu_\Gamma(v)-\mu_n(v)|<\varepsilon$,
     \item[{(b)}]  $|\mu_\Gamma(\sigma(v))-\mu_n(\sigma(v))|<\varepsilon$,
     \item[{(c)}]  $\frac{|F_n\cdot \sigma\Delta F_n|}{|F_n|}<\varepsilon$.
\end{enumerate}
It follows that, for $n\in A$
\begin{align*}
    \mu_\Gamma(\sigma(v))&\leq \mu_n(\sigma(v))+\varepsilon\\
    &=\frac1{|F_n|}\sum_{\tau \in F_n}\mu(\tau\sigma(v))+\varepsilon\\
    &=\frac1{|F_n|}\Big(\sum_{\tau \in F_n}\mu(\tau\sigma(v))+\sum_{\tau \in F_n}\mu(\tau(v))-\sum_{\tau \in F_n}\mu(\tau(v))\Big)+\varepsilon\\
    &\leq \frac1{|F_n|}\sum_{\tau \in F_n}\mu(\tau(v))+\frac{|F_n\cdot\sigma\Delta F_n|}{|F_n|}\|\mu\|_\infty+\varepsilon\\
    &\leq \mu_\Gamma(v)+(2+\|\mu\|_\infty)\varepsilon.
\end{align*}
As $\varepsilon>0$ is arbitrary, it follows that $\mu_\Gamma(\sigma(v))\leq \mu_\Gamma(v)$. Since the argument holds for any $\sigma\in \Gamma$, we get that 
$$
\mu_\Gamma(\sigma(v))= \mu_\Gamma(v),
$$
for every $\sigma\in \Gamma$ and every $v\in V_G$, as claimed.
\end{proof}

Recall that a graph $G$ is called vertex transitive, if for every pair of vertices $v,v'\in V_G$ there exists $\sigma\in \aut(G)$ such that $v'=\sigma(v)$. As a direct consequence of Theorem \ref{t:symmetric} we get:

\begin{cor}\label{c:vertextrans}
Let $G$ be a vertex transitive graph.
\begin{enumerate}
    \item[{(i)}] Suppose $\aut(G)$ is amenable. Then, $G$ supports a bounded doubling measure if and only if the counting measure $|\cdot|$ is doubling.
    \item[{(ii)}]  If $G$ is finite, then $C_G=C_{|\cdot|}$
\end{enumerate}
\end{cor}

Let us see next that property $(i)$ holds actually under more general assumptions. 

\begin{prop}
Let $G$ be a vertex transitive graph. $G$ supports a doubling measure if and only if $|\cdot|$ is doubling.
\end{prop}

\begin{proof}
Suppose $\mu$ is a doubling measure on $G$. A moment thought reveals that for every $v\in V_G$, $r>0$ and $\lambda\geq1$
$$
\mu(B(v,\lambda r))\leq C_\mu \lambda^{\log_2 C_\mu} \mu (B(v,r)).
$$
It follows (cf. \cite{VK}) that there is a constant $D$ (depending only on $C_\mu$), such that the cardinality of a set of $r$-separated points in $B(v,\lambda r)$ is always bounded by $D\lambda^{\log_2 C_\mu}$. In particular, this implies that for every $x\in V_G$, we can find a set $F\subset V_G$ with $|F|\leq D C_\mu$ such that
$$
B(x,2r)\subset \bigcup_{v\in F} B(v,r).
$$

Moreover, since $G$ is vertex transitive, it follows that $|B(v,r)|=|B(x,r)|$ for every $v\in V_G$ and $r>0$. In particular we have
$$
|B(x,2r)|\leq |F||B(x,r)|\leq D C_\mu|B(x,r)|,
$$
which shows that $|\cdot|$ is doubling.
\end{proof}

It should be noted that the fact that $\aut(G)$ is amenable, does not necessarily imply that there exists a doubling measure on $G$, as the following shows. 

\begin{example}
For every $r\geq 3$ there is a graph $G$ with $\Delta_G=r+1$ such that $\aut(G)=\{id_{V_G}\}$ and $G$ does not support any non-trivial doubling measure.
\end{example}

\begin{proof}
For $r\geq3$, let $T_r$ denote the infinite $r$-homogeneous regular tree and let $(x_i)_{i\in\mathbb N}$ be an enumeration of its vertices. We will consider a graph $G$ whose vertex set $V_G$, consists of all pairs $(x_i,j)$ with $j=0,\ldots,i$, and two vertices $(x_i,j)$, $(x_{i'},j')$ 
are connected by an edge in $E_G$ if $i=i'$ and $|j-j'|=1$ or $i\neq i'$, $j=j'=0$ and $x_i$ is connected to $x_{i'}$ by an edge in $E_{T_r}$.

Let $d_{i,j}$ denote the degree of the vertex $(x_i,j)$. We have
$$
d_{i,j}=\left\{\begin{array}{cl}
    r+1, & \text{ if } j=0, \\
    1, & \text{ if } j=i,\\
    2, & \text{ otherwise.}\\
\end{array}
\right.
$$
Since automorphisms preserve the degree of each vertex, it is easy to check that $\aut(G)=\{id_{V_G}\}$.

On the other hand, in order to see that $G$ does not support a doubling measure it is enough to check that $G$ is not metrically doubling; that is, the number of $s$-separated points in a ball of the from $B(v,2s)$ is not uniformly bounded for $s>0$ (cf. \cite{LuSa}). This follows from the properties of $T_r$. Indeed, note that for every $s>0$ the ball $B(v,2s)$ for $v=(x_0,0)$, contains $r(r-1)^{s-1}$ vertices of the form $(x_{i},0)$ with $d(x_0,x_{i})=s$, and for each of them we can consider another vertex of the form $(x_{j_i},0)$ with $d(x_i,x_{j_i})=s$ and $d(x_0,x_{j_i})=2s$, so that $d(x_{j_i},x_{j_{i'}})>s$ for $i\neq i'$.
\end{proof}

\section{The least doubling constants $C_G$ vs. $C_G^0$}\label{sec5}

Given that the restricted doubling constant $C_G^0$ can be more easily computed by Theorems \ref{thm:doubling spectra} and \ref{thm:doubling spectra infinite} than $C_G$ itself, the purpose of this section is to clarify when $C_G=C_G^0$ and use this to compute $C_G$ for certain families of graphs. Let us begin with a simple observation that characterizes when $C_G=C_G^0$:

\begin{prop}\label{lemachorra}
For a finite graph $G$, let $\mu_0$ denote the measure associated to the Perron eigenvector of the adjacency matrix. The following are equivalent:
\begin{enumerate}
\item[{(i)}] $C_G=C_G^0$.
\item[{(ii)}] $C_{\mu_0}=C_{\mu_0}^0$.
\end{enumerate}
\end{prop}

\begin{proof}
Suppose first $C_G=C_G^0$, and let $\mu\in DM(G)$. We have that
$$
C_{\mu}^0\leq C_\mu=C_G=C_G^0\leq C_{\mu}^0.
$$
Hence, we have that 
\begin{equation*}
C_{\mu}^0= C_\mu=C_G=C_G^0.
\end{equation*}
In particular, by Theorem \ref{thm:doubling spectra}, it follows that $\mu=\alpha\mu_0$ for some scalar $\alpha>0$. Thus, $C_{\mu_0}^0= C_{\mu_0}$ as claimed.

Conversely, if $C_G^0<C_G$, then we have
$$
C_{\mu_0}^0=C_G^0<C_G\leq C_{\mu_0},
$$
as claimed.
\end{proof}

In \cite{ST} we computed the least doubling constant of some families of finite graphs, including complete graphs $K_n$, star graphs $S_n$ and cycles $C_n$:
$$
C_{K_n}=n, \quad \quad C_{S_n}=1+\sqrt{n}, \quad \quad C_{C_n}=3.
$$
In \cite{DST} we have analyzed the constants of path graphs $L_n$, $\mathbb N$, and $\mathbb Z$ and have shown that
$$
\lim_{n\rightarrow\infty} C_{L_n}=3=C_{\mathbb Z}=C_{\mathbb N},
$$
where $L_n$ denotes the path or linear graph of $n$ vertices. It was also shown in \cite{DST} that $C_{L_n}=C_{L_n}^0$ if and only if $n\leq 8$.

In order to expand this list, we will need the following result:

\begin{prop}\label{p:diam2}
Let $G$ be a graph with $\diam(G)=2$. For every measure $\mu$ on $G$ we have $C_\mu=C_\mu^0$. In particular,  $C_G=C_G^0$.
\end{prop}

\begin{proof}
Since $\diam(G)=2$, for every $v\in V_G$ we have that 
\begin{equation}\label{bola2}
V_G=B(v,2)=\bigcup_{w\in B(v,1)}B(w,1).
\end{equation}
Hence, if $\mu$ is any measure on $G$, then for every $v\in V_G$, by \eqref{bola2} and Lemma \ref{l:holder}, we have that
$$
\frac{\mu(B(v,2))}{\mu(B(v,1))}\leq\frac{\sum_{w\in B(v,1)}\mu(B(w,1))}{\mu(B(v,1))}=\frac{\sum_{w\in B(v,1)}\mu(B(w,1))}{\sum_{w\in B(v,1)}\mu(w)}\leq \max_{w\in B(v,1)}\Big\{\frac{\mu(B(w,1))}{\mu(w)}\Big\}.
$$
It follows that $C_\mu= C_\mu^0$, and in particular $C_G=C_G^0$.
\end{proof}

\subsection{Significant families of graphs with diameter 2}\label{subsec}
\begin{enumerate} [leftmargin=*]
\item[{(i)}] {\em Friendship graphs}: The celebrated Friendship Theorem \cite{ERS}, states that the only graphs with the property that every two vertices have exactly one neighbor in common are the friendship graphs $F_n$ ($2n+1$ vertices, obtained by joining $n$ copies of $C_3$ with a common vertex). Let $v_1$ be the vertex of degree $2n$, and $(v_i)_{i=2}^{2n+1}$ the vertices of degree 3. For every $n\in \mathbb N$ we have $C_{F_n}=C_{F_n}^0=1+\frac12(1+\sqrt{1+8n})$. Indeed, by Proposition \ref{p:diam2}, Theorem \ref{thm:doubling spectra} and \cite[Proposition~2.2]{AJR}, we have that 
$$
C_{F_n}= C_{F_n}^0=1+r(A_{F_n})=1+\frac12(1+\sqrt{1+8n}). 
$$
Moreover, 
$$
\mu(v_i)=
\left\{
\begin{array}{ccl}
4n,  &   & i=1,  \\
1+\sqrt{1+8n},  &   &   2\leq i\leq 2n+1,
\end{array}
\right.
$$
defines the unique doubling minimizer (up to constants).

\item[{(ii)}] {\em Wheel graphs}: Let $W_n$ denote the wheel graph with $n$ vertices, being $v_1$ a vertex of degree $n-1$ and the remaining vertices $(v_i)_{i=2}^n$ of degree 3. For every $n\in \mathbb N$ we have $C_{W_n}=C_{W_n}^0=2+\sqrt{n}$. Indeed, by Proposition \ref{p:diam2}, Theorem \ref{thm:doubling spectra} and \cite[5.6]{M}, it follows that
$$
C_{W_n}=C_{W_n}^0=1+r(A_{W_n})=2+\sqrt{n}.
$$
Moreover,  
$$
\mu(v_i)=
\left\{
\begin{array}{ccl}
n-1,  &   & i=1,  \\
1+\sqrt{n},  &   &   2\leq i\leq n,
\end{array}
\right.
$$
defines the unique doubling minimizer (up to constants).


\item[{(iii)}] {\em Complete bipartite graphs}: $C_{K_{m,n}}=1+\sqrt{mn}$

\item[{(iv)}] {\em Cocktail party graph}: $C_{R_{2n-2,2n}}=2n-1$. The {\em cocktail party} is the unique regular graph with $2n$ vertices of degree $2n-2$. 

\item[{(v)}] {\em Petersen graph}: $C_{R_{3,10}}=4$.

\item[{(vi)}] {\em Hoffman-Singleton graph}: $C_{R_{7,50}}=8$.

\item[{(vii)}] {\em Clebsch graph}: $C_{R_{4,16}}=5$.

\end{enumerate}

Items (iv)-(vii) correspond to regular graphs. Recall that a regular graph of degree $k$ is a graph all of whose vertices have degree $k$. Let $\Delta(G)$ denote the maximum degree of $G$. It is well-know (cf. \cite[Theorem 6]{M}) that for connected graphs, $\lambda_1(A_G)=\Delta(G)$ if and only if $G$ is a regular graph. In particular, by Theorem \ref{thm:doubling spectra}, every $k$-regular graph $G$ satisfies $C_G^0=k+1$.

\begin{figure}[h]
\begin{center}
\includegraphics[width=\textwidth]{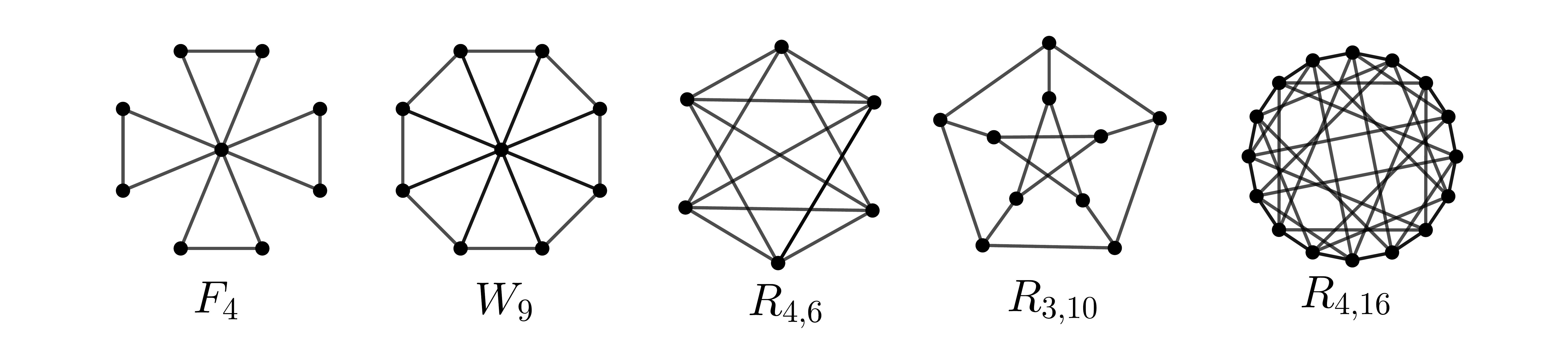}
\caption{Some of the graphs in Subsection~\ref{subsec}.}
\end{center}
\end{figure}


\subsection{Graphs with $C_G>C_G^0$}

Let us consider the following three-legs graph $T$: this is a tree with one vertex $v_1$ of degree 3, whose 3 neighbors $v_2,v_3,v_4$ have degree 2, and each of these is connected to a leave $v_5,v_6,v_7$ (see Figure \ref{Tripar}). It can be easily checked that the Perron eigenvector of $A_T$ yields the measure 
$$
\mu(v_i)=
\left\{
\begin{array}{ccl}
3,  &   & i=1,  \\
2,  &   &   2\leq i\leq 4,\\
1,  &   & 5\leq i\leq 7.
\end{array}
\right.
$$
Since 
$$
C_\mu\geq\frac{\mu(B(v_5,3))}{\mu(v_5)}=\frac{10}{3}>3=C_\mu^0,
$$
by Proposition \ref{lemachorra}, it follows that $C_T>C_T^0$. The following provides an accurate computation of $C_T$. Although the proof is completely elementary, we have decided to include the details as an illustration of the method that can be used for computing $C_G$ for other graphs.

\begin{figure}[h]
\begin{center}
\includegraphics[width=.4\textwidth]{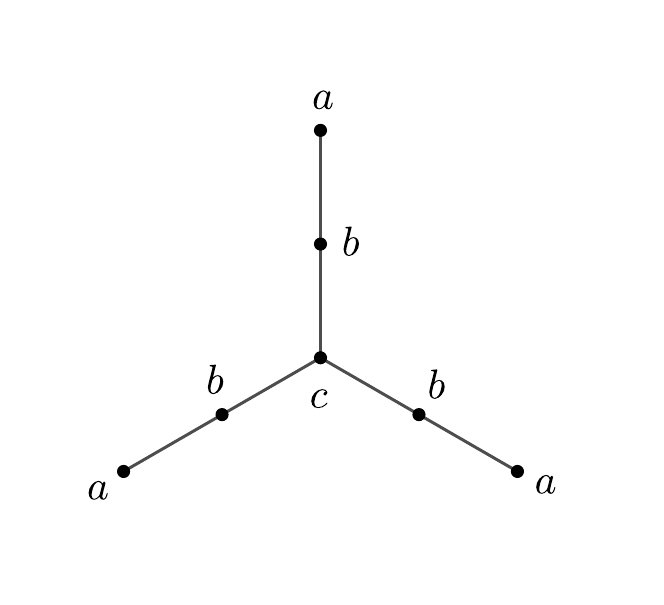}
\caption{The three-legs graph $T$.}\label{Tripar}
\end{center}
\end{figure}

\begin{prop}\label{p:tripar}
Let $T$ be the  three-legs graph given above. It holds that $C_T-1$ coincides with the largest root of the polynomial $x^3+x^2-5x-3$. In particular,
$$
\sup_k C_T^k\leq3 <C_T\approx  3.0861.
$$
\end{prop}

\begin{proof}
As $\diam(G)=4$ we have to show $C_G^k\leq 3$ for $k=0,1,2$. By Lemma \ref{l:finite}, we can assume all measures involved in the computations are symmetric; that is, we can set $a$ for the value of the measure on each leaf, $b$ for its value on the vertices of degree 2 and $c$ for that of the vertex of degree 3, as in Figure \ref{Tripar}.

Note that
\begin{equation*}
C_T^0=\inf_{a,b,c}\sup\Big\{\frac{a+b}{a},\frac{a+b+c}{b},\frac{3b+c}{c}\Big\},
\end{equation*}
so if we take $a=1$, $b=2$, $c=3$, it follows that $C_T^0\leq 3$.

Similarly, we also have
\begin{equation*}
C_T^1=\inf_{a,b,c}\sup\Big\{\frac{a+3b+c}{a+b},\frac{3a+3b+c}{a+b+c},\frac{3a+3b+c}{3b+c}\Big\},
\end{equation*}
so taking $a=b=1$ and $c=2$ we get $C_T^1\leq 3$.


Finally, note that 
\begin{equation*}
C_T^2=\inf_{a,b,c}\sup\Big\{\frac{3a+3b+c}{a+b+c},1\Big\}\leq 3.
\end{equation*}
Now, let us compute explicitly $C_T$. 
To give an upper bound let us consider the weights
$$
a=1\quad b= (1+r)/2\quad c= (-2 + r + r^2)/2,
$$
where $r\approx 2.0861$ is the largest zero of the polynomial $x^3+x^2-5x-3=0$. It is easy to check that for these particular weights $C_T\leq1+r.$ 

On the other hand, let $\overline{a}, \overline{b}, \overline{c}$ denote the weights that give $C_T$. We know that 
\begin{equation*}
\begin{split}
\sup\left\{\dfrac{\overline{a}+\overline{b}+\overline{c}}{\overline{b}},\dfrac{3\overline{b}+\overline{c}}{\overline{c}},\dfrac{\overline{a}+3\overline{b}+\overline{c}}{\overline{a}+\overline{b}}\right\}\leq C_T.
\end{split}
\end{equation*}
Thus, if we denote $x_1=\overline{a}+\overline{b}$, $x_2=\overline{b}$ and $x_3=\overline{c}$, then $C_T$ satisfies
$$
\left\{
\begin{array}{lc}
 (i)\quad x_1+2x_2+x_3&\leq C_T x_1  \\ 
 (ii) \quad x_1+x_3&\leq C_Tx_2\\
  (iii)\quad  3x_2+x_3&\leq C_Tx_3
  \end{array}
\right.
$$
Now, an appropriate recursion of this inequalities yields
\begin{equation*}
\begin{split}
&2x_2+x_3\stackrel{(i)}{\leq}x_1 (C_T-1) \stackrel{(ii)}{\leq} (C_Tx_2-x_3) (C_T-1)\Rightarrow 0\leq x_2(C_T^2-C_T-2)-x_3 C_T\\
&\stackrel{(iii)}{\Rightarrow}0\leq \frac{x_3}{3}(C_T-1)(C_T^2-C_T-2)-x_3 C_T \Rightarrow C_T^3-2C_T^2-4C_T+2\geq 0.
\end{split}
\end{equation*}
Observe that
$$
x^3-2x^2-4x+2=(x-1)^3+(x-1)^2-5(x-1)-3
$$
and 
$$
\{x:(x-1)^3+(x-1)^2-5(x-1)-3=0\}=\{1+x:x^3+x^2-5x-3=0\}.
$$
Since $C_T\geq 3$, we conclude that $C_T\geq 1+r$.
\end{proof}

Proposition \ref{p:diam2} cannot be extended to graphs with diameter larger than 2, as the following shows.

\begin{example}
The {\em Doyle graph} (also known as the Holt graph \cite{Holt}) $\mathcal D$ is a vertex transitive 4-regular graph of diameter $3$ with $27$ vertices which satisfies
$$
C_{\mathcal D}>C_{\mathcal D}^0.
$$
\end{example}

Indeed, because $\mathcal D$ is a vertex transitive graph, if follows from Corollary \ref{c:vertextrans} that 
$$
C_{\mathcal D}=C_{|\cdot|}=\sup\left\{\frac{|B(v,3)|}{|B(v,1|)}, \frac{|B(v,1)|}{|B(v,0)|}\right\}=\sup\left\{\frac{n}{k+1}, \frac{k+1}{1}\right\}=\sup\left\{\frac{27}{5}, 5\right\}=\frac{27}{5}.
$$
On the other hand, being a $4$-regular graph, we have $C_{\mathcal D}^0=5,$ and the claim follows.

\begin{figure}[h]
\begin{center}
\includegraphics[width=.5\textwidth]{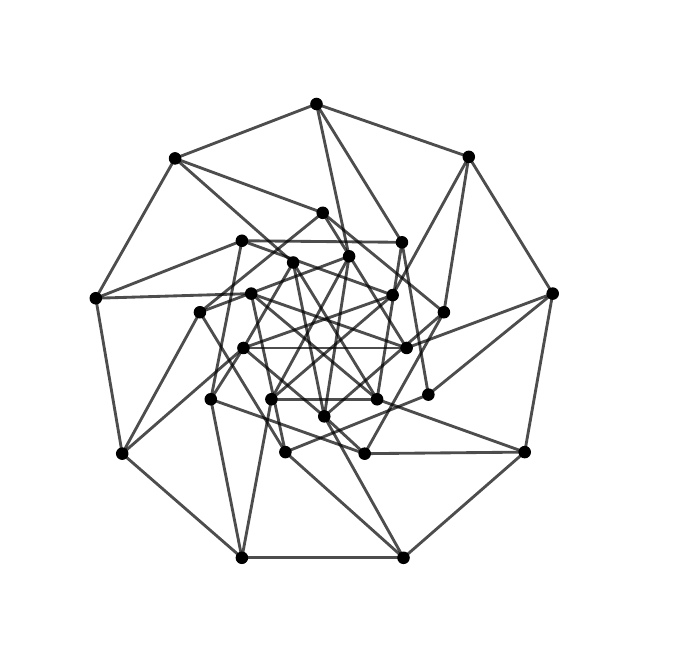}
\caption{Doyle graph. }
\end{center}
\end{figure}

\section{A characterization of graphs with doubling constant $C_G\leq3$}\label{sec6}

Motivated by Smith's description of graphs with spectral radius no greater than 2 \cite{S}, our purpose in this Section is to classify those graphs $G$ with doubling constant $C_G\leq 3$. Note that $C_G=2$ only in the case when the graph consists of 2 vertices. Let us begin with some simple facts steaming directly from the definition of $C_G^0$:

\begin{lem}\label{l:localmax}
Given a graph $G$ and $\mu$ a measure on $G$. Suppose $v\in V_G$ has degree $d_v\geq 3$ and let $(v_i)_{i=1}^{d_v}$ be the set of its neighbors. Let $a_0=\mu(v)$, and $a_i=\mu(v_i)$ for $1\leq i\leq d_v$. If $a_0\leq a_i$ for some $1\leq i\leq d_v$, then $C^0_\mu\geq3$.
\end{lem}

\begin{proof}
Suppose first that $\sum_{j\neq i}a_i\leq a_0$, then there exists $j\neq i$ such that $a_j\leq \frac{a_0}{d_v-1}$. In this case, we have
$$
C^0_\mu\geq \frac{\mu(B(v_j,1))}{\mu(B(v_j,0))}\geq\frac{a_j+a_0}{a_j}\geq d_v\geq3.
$$
Otherwise, suppose now that $\sum_{j\neq i}a_i> a_0$. In that case, we have
$$
C^0_\mu\geq\frac{\mu(B(v,1))}{\mu(B(v,0))}\geq\frac{\sum_{i=0}^{d_v} a_i}{a_0}>3.
$$
\end{proof}

\begin{prop}\label{CG3}
Let $G$ be a graph. Under any of the following conditions we have that $C^0_G\geq3$:
\begin{enumerate}
  \item[{(i)}] $G$ contains a cycle.
  \item[{(ii)}] $G$ contains a vertex of degree strictly larger than 3.
  \item[{(iii)}] $G$ contains 2 vertices of degree at least 3.
\end{enumerate}
\end{prop}

\begin{proof}
To prove \textit{(i)}, let $\{v_i:i=1,\ldots,m\}\subset V_G$ be an $m$-cycle ($m\geq3$); that is, $\{v_1, v_m\},\,\{v_i, v_{i+1}\}\in E_G$ for $i=1\ldots,m-1$. For any measure $\mu$ on $G$, let $a_i=\mu(v_i)$ for $i=1,\ldots,m$. Let $a_{i_0}=\min\{ a_i:1\leq i\leq m\}$. It follows that
$$
C^0_{\mu}\geq \frac{\mu(B(v_{i_0},1))}{\mu(B(v_{i_0},0))}\geq \frac{a_{i_0-1}+a_{i_0}+a_{i_0+1}}{a_{i_0}}\geq 3.
$$
Since $\mu$ is arbitrary, we have $C^0_\mu\geq3$.

 \textit{(ii)}: Suppose $v\in V_G$ has degree $d_v\geq 4$. Let $(v_i)_{i=1}^{d_v}\subset V_G$ be the neighbors of $v$. For any measure $\mu$ on $G$, let $a_0=\mu(v)$ and $a_i=\mu(v_i)$ for $i=1,\ldots,d_v$. Suppose first that
$$
\sum_{i=1}^{d_v} a_i\geq 2a_0.
$$
In this case, we have
$$
C^0_\mu\geq \frac{\mu(B(v,1))}{\mu(B(v,0))}=\frac{\sum_{i=0}^{d_v} a_i}{a_0}\geq 3.
$$
Now, suppose on the other hand that
$$
\sum_{i=1}^{d_v} a_i< 2a_0.
$$
Then, necessarily, there is $1\leq j\leq d_v$ such that $a_j<\frac{2a_0}{d_v}$. In this case, as $d_v\geq4$, we have
$$
C^0_\mu\geq \frac{\mu(B(v_j,1))}{\mu(B(v_j,0))}\geq\frac{a_j+a_0}{a_j}>\frac{d_v}{2}+1\geq3.
$$
In either case, it follows that $C^0_G\geq3$.

 \textit{(iii)}: By  \textit{(i)} and \textit{(ii)}, we can assume that $G$ is a tree with maximal degree $\Delta_G=3$. Hence, without loss of generality, we can assume that $G$ contains two vertices of degree 3, $v_1,v_2$, which are connected by a path of length $p\in\mathbb N$, $\{w_{i-1},w_i\}$ for $1\leq i\leq p$, with all the vertices distinct from $v_1$ and $v_2$ having degree 2. Note that $v_1$ and $v_2$ could be neighbors; that is, $p=0$, in which case, the same proof works.

Given any measure $\mu$ on $G$, let $a_0=\mu(v_1)$, $a_i=\mu(w_i)$ for $i=1,\ldots,p$, and $a_{p+1}=\mu(v_2)$. By Lemma \ref{l:localmax}, we can assume $a_0>a_1$ and $a_p<a_{p+1}$. It follows, that there must be some $1\leq i\leq p$ such that $a_i\leq \min\{a_{i-1},a_{i+1}\}$. Therefore, we have
$$
C^0_\mu\geq\frac{\mu(B(w_i,1))}{\mu(B(w_i,0))}\geq\frac{a_{i-1}+a_i+a_{i+1}}{a_i}\geq3.
$$
\end{proof}

\begin{remark}
Alternatively, Proposition \ref{CG3} can be proved using the monotonicity of $C_G^0$ (which is a consequence of Theorem \ref{thm:doubling spectra}) and the fact that the graphs $S_n$ for $n\geq 4$, $C_n$ and $\widehat D_n$ satisfy $C_G^0\geq3$. However, we included the proof above as it is completely elementary.
\end{remark}

If a graph satisfies $C_G\leq3$, Proposition \ref{CG3} implies that $G$ must be a tree with at most one vertex of degree 3, and all the remaining vertices have degree 1 or 2. However, the list of graphs with $C_G\leq3$ is even smaller: By Theorem \ref{thm:doubling spectra}, we have that 
$$
1+\lambda_1(A_G)=C^0_G\leq C_G\leq 3,
$$
so $\lambda_1(A_G)\leq 2$. The finite graphs with this property have been considered in \cite{S} and \cite{CG} (see Figure \ref{dynkin}).

\begin{figure}[h]
\begin{center}
\includegraphics[width=\textwidth]{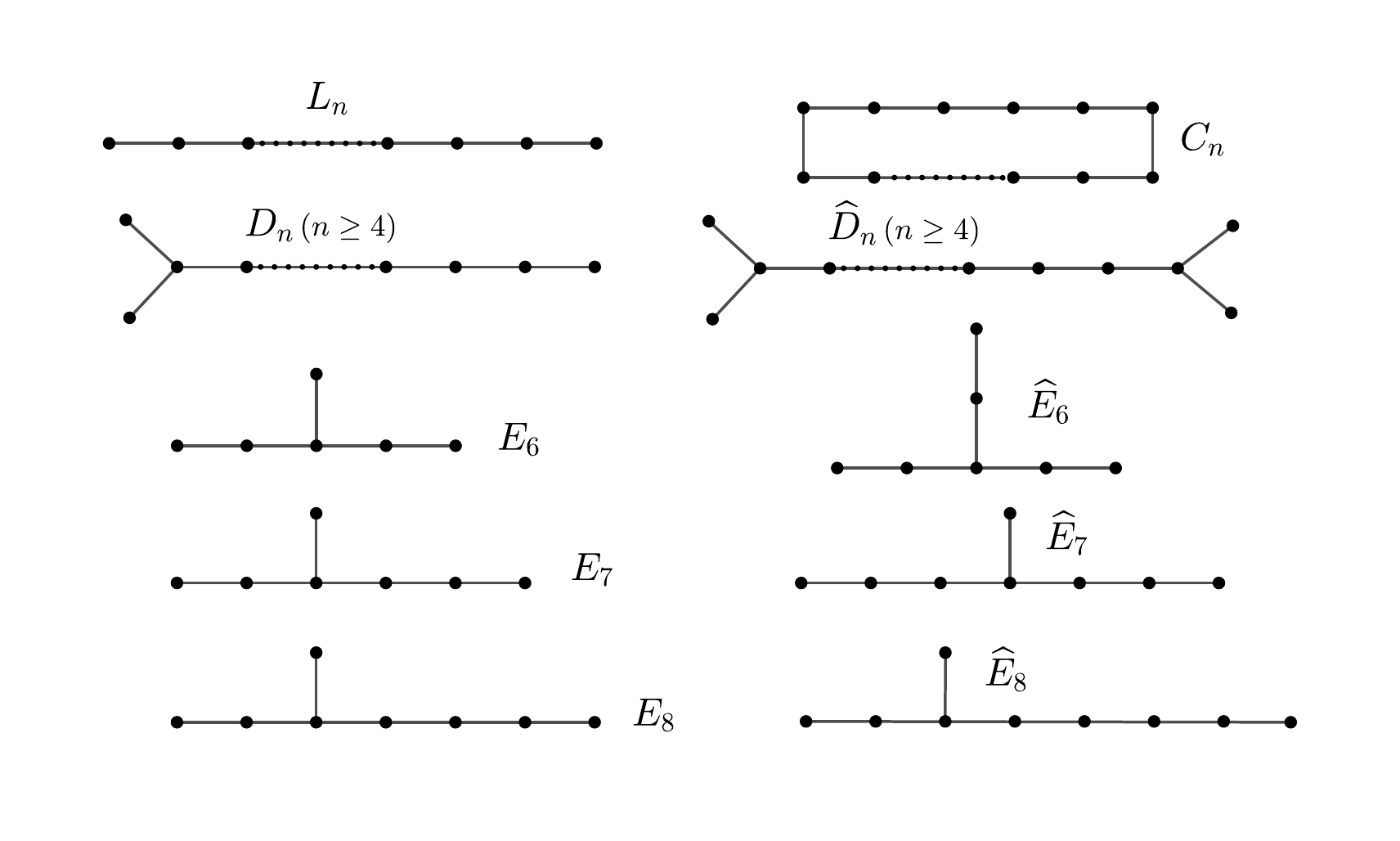}
\caption{Graphs with small doubling constant.}\label{dynkin}
\end{center}
\end{figure}


In particular, we have 
$$
C^0_{D_n}=1+2\cos{\frac{\pi}{2(n-1)}}\qquad C^0_{L_n}=1+2\cos{\frac{\pi}{n+1}},
$$
$$
C^0_{E_6}=1+2\cos{\frac{\pi}{12}},\quad C^0_{E_7}=1+2\cos{\frac{\pi}{18}},\quad C^0_{E_8}=1+2\cos{\frac{\pi}{30}},
$$
and
\[
3=C^0_{\widehat{E}_6}=C^0_{\widehat{E}_7}=C^0_{\widehat{E}_8}=C^0_{C_n}=C^0_{\widehat{D}_n},
\]
where $n=|V_{D_n}|=|V_{L_n}|=|V_{\widehat{D}_n}|=|V_{C_n}|$.

\begin{cor}\label{c:finiteCG3}
The only finite graphs with $C_G\leq 3$ are $E_6$, $E_7$, $L_n$, $D_n$, $C_n$ and $\widehat{D}_n$.
\end{cor}

\begin{proof}
In \cite{DST} it is proved that $C_{L_n}<3$ and in \cite{ST} it was checked that $C_{C_n}=3.$ 

On the other hand, consider in $D_n$ the measure $\mu$ that gives $1$ to the two leaves and $2$ to the rest of the vertexes. It is easy to see that $C_{D_n}\leq C_{\mu}=3$. 

Let us consider in $E_6$ the measure $\mu$ and in $E_7$ the measure $\nu$ as in Figure \ref{E6E7}. One can check that $C_{E_6}\leq C_{\mu}<3$ and $C_{E_7}\leq C_{\nu}<3$. To see that $C_{E_8}>3$, one can proceed similarly as we did in Proposition \ref{p:tripar}. If we denote by $a_i$ ($1\leq i\leq 8$) a sequence of optimal weights to compute $C_{E_8}$ (see Figure \ref{E8}) one can check that 
\begin{equation*}
\begin{split}
\sup&\Big\{\frac{a_1+a_2+a_3+a_4+a_8}{a_1+a_2}, \frac{a_1+a_2+a_3}{a_2},\frac{a_2+a_3+a_4}{a_3}, \frac{a_3+a_4+a_5}{a_4}, \\&\frac{a_4+a_5+a_6}{a_5}, \frac{a_5+a_6+a_7}{a_6}, \frac{a_2+a_3+a_4+a_5+a_6+a_7+a_8}{a_5+a_6+a_7},\frac{a_3+a_8}{a_8} \Big\}\leq C_{E_8}=x,
\end{split}
\end{equation*}
implies
\begin{equation}\label{eqpolyE8}
2 x - 8 x^2 + 14 x^3 - 10 x^4 - 4 x^5 + 11 x^6 - 6 x^7 + x^8\geq 0.
\end{equation}
Since $x\geq 2$, $x\geq r\approx 3.02058$ where $r$ is the biggest root of the polynomial in \eqref{eqpolyE8}.
\begin{figure}[h]\label{E8}
\begin{center}
\includegraphics[width=.6\textwidth]{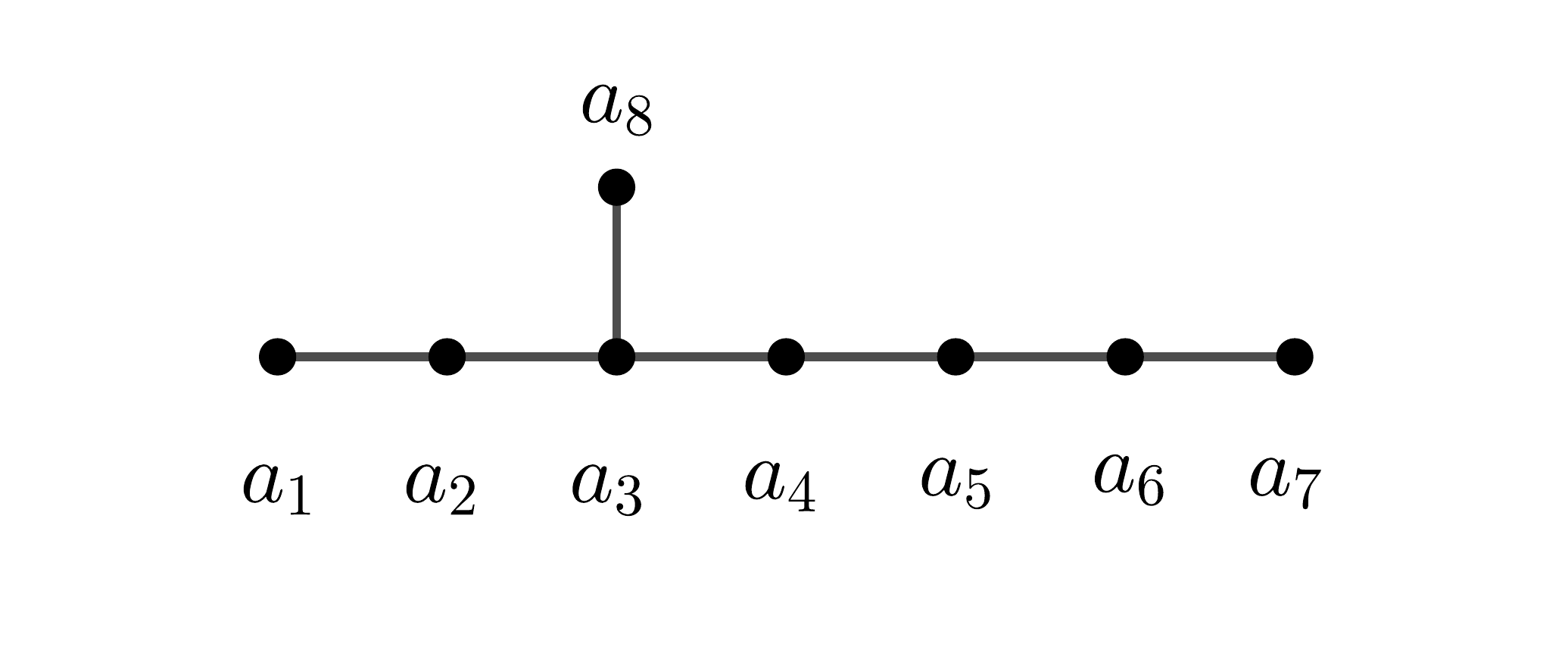}
\caption{Graph $E_8$.}
\end{center}
\end{figure}

Now, let $\mu_0$ denote the measure associated to the Perron eigenvector of the adjacency matrix (see Figure \ref{dynkin}). 
If $G$ is $\widehat{E}_6, \widehat{E}_7$ or $\widehat{E}_8$ it can be numerically computed that $C_{\mu_0}>C_{\mu_0}^0$ so, by Proposition \ref{lemachorra}, $C_{\widehat{E}_6}>3$, $C_{\widehat{E}_7}>3$ and $C_{\widehat{E}_8}>3$ (notice that in Proposition \ref{p:tripar} we computed the exact value of $\widehat{E}_6$). 

Finally, a similar argument as in the proof that the counting measure $\lambda$ on $L_n$ satisfies $C_{\lambda}=3$ can be used to show that $C_{\widehat{D_n}}=C^0_{\widehat{D_n}}=3$.
%
%

\begin{figure}[h]\label{E6E7}
\begin{center}
\includegraphics[width=\textwidth]{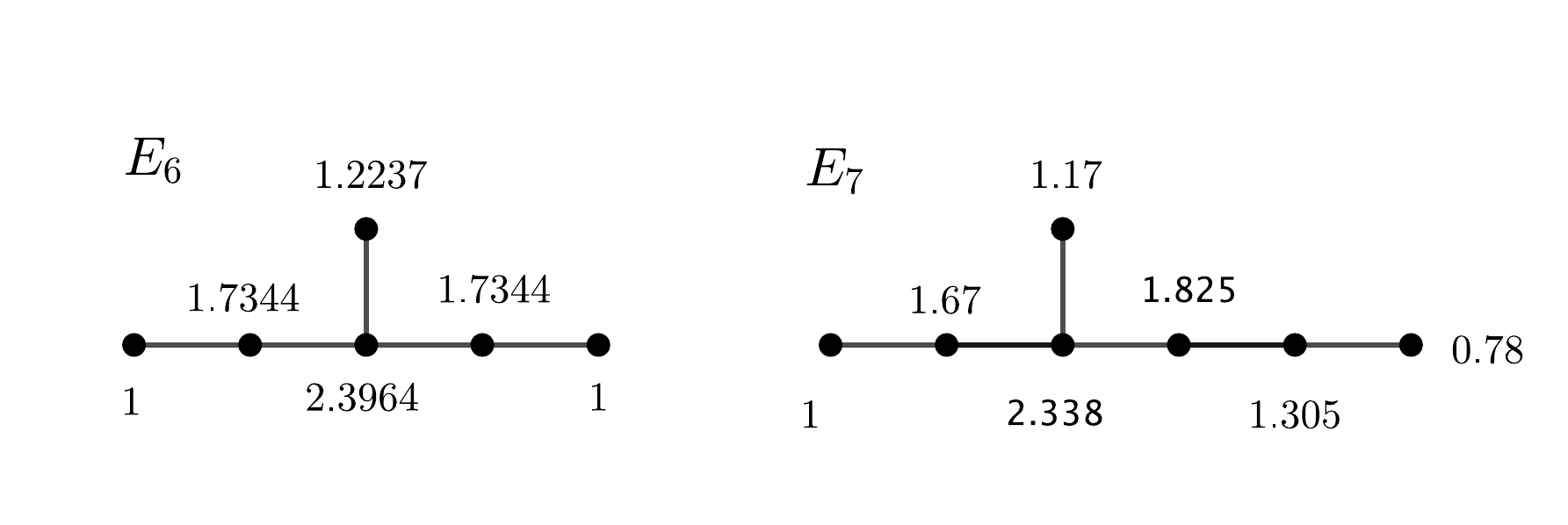}
\caption{Graphs $E_6$ and $E_7$.}
\end{center}
\end{figure}
\end{proof}

\begin{cor}\label{c:infiniteCG3}
The only infinite graphs with $C_G\leq 3$ are $\mathbb N$, $\mathbb Z$, and $D_{\infty}$.
\end{cor}

\begin{proof}
Suppose $G$ is an infinite graph with $C_G\leq3$. In particular, $C_G^0\leq 3$, so by Proposition \ref{p:C0supsubgraph}, every finite subgraph $F\subset G$ must have $C^0_F\leq 3$. However, by Corollary \ref{c:finiteCG3}, the only finite graphs of arbitrary size are $D_n$ and $L_n$. This makes $\mathbb N$, $\mathbb Z$, and $D_{\infty}$ the only possible infinite graphs with constant no greater than 3. 

The fact that actually $C_\mathbb N=C_\mathbb Z=3$ has been proved in \cite{DST}. As for $D_\infty$, let $(v_j)_{j\in\mathbb N}$ be an enumeration of the vertices, so that $v_1$ and $v_2$ correspond to the leaves, $v_3$ is the only vertex of degree 3, and the remaining vertices have degree 2. It is enough to consider the measure given by
$$
\mu(v_j)=
\left\{
\begin{array}{ccl}
1,  &   & j=1,2,  \\
2,  &   &   j\geq 3,
\end{array}
\right.
$$
and a straightforward computation yields $C_{D_\infty}\leq C_\mu=3$.
\end{proof}

%

\medskip

\noindent
\textbf{Acknowledgment.} The first author would like to thank Professor Miguel \'Angel Sama (UNED),  for his valuable help with the numerical calculations (and some enlightening conversations)  involved in the preliminary versions of this manuscript. She would also like to thank Professor Jeremy T. Tyson (University of Illinois at Urbana-Champaign) for a fruitful conversation regarding this topic.


\begin{thebibliography}{99}

\bibitem{AJR} A. Abdollahi, S. Janbaz, and M. Reza Oboudi,  \textit{Graphs cospectral with a friendship graph or its complement,} Transactions on Combinatorics 2 (2013), 37--52.

\bibitem{AA} Y. A. Abramovich and C. D. Aliprantis, \textit{An invitation to operator theory}, Graduate Studies in Mathematics. 50. Providence, RI: American Mathematical Society, 2002. 

\bibitem{AL} S. Adams and R. Lyons, \textit{Amenability, Kazhdan's property and percolation for trees, groups and equivalence relations,} Israel J. Math. 75 (1991), no. 2-3, 341--370.

\bibitem{BR} M. Behzad and H. Radjavi, \textit{Chromatic numbers of infinite graphs}, J. Combinatorial Theory Ser. B 21 (1976), no. 3, 195--200.


\bibitem{BM} J. A. Bondy and  U. S. R. Murty, \textit{Graph theory}, Graduate Texts in Mathematics, 244. Springer, New York, 2008.


\bibitem{CW} R. R. Coifman and G. Weiss, \textit{Analyse harmonique non-commutative sur certains espaces homog\`enes. \'Etude de certaines int\'egrales singuli\`eres, } Lecture Notes in Math. 242. Springer-Verlag, Berlin-New York, 1971. 

\bibitem{CG} D. Cvetkovi\'c and I. Gutman, \textit{On spectral structure of graphs having the maximal eigenvalue not greater than two.} Publ. Inst. Math., Nouv. S\'er. 18 (1975), 39--45. 

\bibitem{CRS} D. Cvetkovi\'c, P. Rowlinson, and S. Simi\'c, \textit{An introduction to the theory of graph spectra,} London Mathematical Society Student Texts, 75. Cambridge University Press, Cambridge, 2010.

\bibitem{DST} E. Durand-Cartagena, J. Soria, and P. Tradacete, \textit{The least doubling constant of a path graph} (preprint).

\bibitem{ERS} P. Erd\"os, A. R\'enyi, and V. T. S\'os, \textit{On a problem of graph theory}, Studia Sci. Math. Hungar. 1 (1966), 215--235.

\bibitem{Folner} E. F\o lner, \textit{On groups with full Banach mean value}, Math. Scand. 3 (1955), 243--254.

\bibitem{Friedland} S. Friedland, \textit{Characterizations of the spectral radius of positive operators,} Linear Algebra Appl. 134 (1990), 93--105.


\bibitem{Holt} D. F. Holt, \textit{A graph which is edge transitive but not arc transitive,} J. Graph Theory 5 (1981), no. 2, 201--204.

\bibitem{LuSa} J. Luukkainen and E. Saksman, \textit{Every complete doubling metric space carries a doubling measure,} Proc. Amer. Math. Soc. 126 (1998), no. 2, 531--534.



\bibitem{M} P. V. Mieghen, \textit{Graph Spectra for Complex Networks,} Cambridge University Press, 2011.

\bibitem{Mohar} B. Mohar, \textit{The spectrum of an infinite graph,} Linear Algebra Appl. 48 (1982), 245--256.
 
\bibitem{S} J. H. Smith, \textit{Some properties of the spectrum of a graph} (Proc. Calgary Internat. Conf., 1969),  Gordon and Breach, New York (1970), 403--406.

\bibitem{SW} P. M. Soardi and W. Woess, \textit{Amenability, unimodularity, and the spectral radius of random walks on infinite graphs}, Math. Z. 205 (1990), no. 3, 471--486.

\bibitem{ST} J. Soria and P. Tradacete, \textit{The least doubling constant of a metric measure space,} Ann. Acad. Sci. Fenn. Math. 44 (2019), 1015--1030. 

\bibitem{ST2} J. Soria and P. Tradacete, \textit{Geometric properties of infinite graphs and the Hardy-Littlewood maximal operator}, J. Anal. Math. 137 (2019), no. 2, 913--937.

\bibitem{VK} A. L. Volberg and S. V. Konyagin, \textit{A homogeneous measure exists on any compactum in $\mathbb R^n$,} (Russian) Dokl. Akad. Nauk SSSR 278 (1984), no. 4, 783--786.

\bibitem{Wie} H. Wielandt, \textit{Unzerlegbare, nicht negative Matrizen,} Math. Z. 52 (1950), 642--648.

\bibitem{Wilf} H. S. Wilf, \textit{The eigenvalues of a graph and its chromatic number,} Lond. Math. Soc. 42 (1967), 330--332. 

\bibitem{Woess} W. Woess, \textit{Amenable group actions on infinite graphs,} Math. Ann. 284 (1989), 251--265.

\end{thebibliography}
\end{document}